\documentclass[preprint,12pt]{elsarticle}
\usepackage{amsthm}
\usepackage{amsmath}
\usepackage{amssymb}
\usepackage{mathrsfs}
\usepackage{enumerate}
\usepackage{ifthen}
\usepackage{graphicx}
\usepackage[T1]{fontenc} %skandit
\newtheorem{theorem}{Theorem}
\newtheorem{lemma}{Lemma}

\newtheorem{problem}{Problem}
\newtheorem{definition}{Definition}
\newtheorem{observation}{Observation}

\journal{}

\begin{document}

\begin{frontmatter}



\title{Distance Magic Index One Graphs}


\author[a]{A V Prajeesh}
\ead{prajeesh\_p150078ma@nitc.ac.in}
\author[a]{Krishnan Paramasivam \corref{cor1}}
%\author[a]{K Paramasivam \corref{cor1}\fnref{fn1}}
\ead{sivam@nitc.ac.in}
\cortext[cor1]{Corresponding author}
%\fntext[fn1]{Student}
\address[a]{Department of Mathematics \\ National Institute of Technology Calicut \\ Kozhikode~\textnormal{673601}, India. }

\begin{abstract}
Let $S$ be a finite set of positive integers. A graph $G=(V(G),E(G))$ is said to be $S$-magic if there exists a bijection $f: V(G) \rightarrow S$ such that for any vertex $u$ of $G$, $\sum_{v\in N_G(u)} f(v)$ is a constant, where $N_G(u)$ is the set of all vertices adjacent to $u$. Let $\alpha(S)$ = $\displaystyle \max_{x\in S} x$. Define $i(G)$ = $\displaystyle \min_S \alpha(S) $, where the minimum runs over all $S$ for which the graph $G$ is $S$-magic. Then $i(G)-|V(G)|$ is called the distance magic index of a graph $G$. In this paper, we compute the distance magic index of graphs $G[\bar{K_{n}}]$, where $G$ is any arbitrary regular graph, disjoint union of $m$ copies of complete multi-partite graph and disjoint union of $m$ copies of graph $C_{p}[\bar{K_{n}}]$, with $m\geq 1$. In addition to that, we also prove some necessary conditions for an regular graph to be of distance magic index one.
\end{abstract}


\end{frontmatter}


\section{Introduction}
\noindent In this paper, we consider only simple and finite graphs. We use $V(G)$ for the vertex set and $E(G)$ for the edge set of a graph $G$. The neighborhood, $N_G(v)$ or shortly $N(v)$ of a vertex $v$ of $G$ is the set of all vertices adjacent to $v$ in $G$. For further graph theoretic terminology and notation, we refer Bondy and Murty \cite{bondy2008graph} and Hammack $et$ $al.$\cite{hammack2011handbook}.
\par
A distance magic labeling of a graph $G$ is a bijection $f : V(G) \rightarrow \{1,...,|V(G)|\}$, such that for any $u$ of $G$, the weight of $u$, $w_G(u) = \sum\limits_{v\in N_{G}(u)} f(v) $ is a constant, say $c$. A graph $G$ that admits such a labeling is called a distance magic graph.
\par
The motivation for distance magic labeling came from the concept of magic squares and tournament scheduling. An equalized incomplete tournament, denoted by $EIT(n, r)$, is a tournament, with $n$ teams and $r$ rounds, which satisfies the following conditions:
\begin{itemize}
	\item [(i)] every team plays against exactly $r$ opponents.
	\item [(ii)] the total strength of the opponents, against which each team plays is a constant.
\end{itemize}
\par
Therefore, finding a solution for an equalized incomplete tournament $EIT(n,r)$ is equivalent to establish a distance magic labeling of an $r$-regular graph of order $n$. For more details, one can refer \cite{froncek2007fair,froncek2006fair}.
\par
The following results provide some necessary condition for distance magicness of regular graphs.
\begin{theorem}\label{oddregular}
	\textnormal{\cite{jinnah,miller2003distance,rao,vilfred}} No $r$-regular graph with $r$-odd can be a distance magic graph.
\end{theorem}
\begin{theorem}\label{2mode4}
	\textnormal{\cite{froncek2006fair}} Let $EIT(n, r)$ be an equalized tournament with an even number $n$ of teams and $r\equiv 2 \mod 4$. Then $n\equiv 0 \mod 4$.
\end{theorem}
\par  In \cite{miller2003distance}, Miller $et$ $al.$ discussed the distance magic labeling of the graph $H_{n,p}$, the complete multi-partite graph with $p$ partitions in which each partition has exactly $n$ vertices, $n \geq 1$ and $p \geq 1$. It is clear that $H_{n,1}$ is a distance magic graph. From \cite{miller2003distance} it is observed that $K_{n}$ is distance magic if and only if $n = 1$ and hence, $H_{1,p}\cong K_{p} $ is not distance magic for all $p\ne 1$. The next result gives a characterization for the distance magicness of $H_{n,p}$.
\begin{theorem}\label{millerHnp}
	\textnormal{\cite{miller2003distance}} Let $n > 1$ and $p > 1$. $H_{n,p}$ has a labeling if and only if either $n$ is even or both $n$ and $p$ are odd.
\end{theorem}
\par Recall a standard graph product (see \cite{hammack2011handbook}). Let $G$ and $H$ be two graphs. Then, the lexicographic product $G \circ H$ or $G[H]$ is a graph with the vertex set $V(G)\times V(H)$. Two vertices $(g,h)$ and $(g',h')$ are adjacent in $G[H]$ if and only if $g$ is adjacent to $g'$ in $G$, or $g=g'$ and $h$ is adjacent to $h'$ in $H$.
\par Miller $et$ $al.$ \cite{miller2003distance} proved the following.
\begin{theorem}\label{millerregularcomp}
	\textnormal{\cite{miller2003distance}} Let $G$ be an arbitrary regular graph. Then $G[\bar{K_n}]$ is distance
	magic for any even $n$.
\end{theorem}
Later, Froncek $et$ $al.$ \cite{froncek2006fair,froncek2011constructing} proved the following results.
\begin{theorem}\label{froncek3}
	\textnormal{\cite{froncek2006fair}} For $n$ even an $EIT(n, r)$ exists if and only if $2 \leq r \leq n-
	2; r \equiv 0 \mod 2$ and either $n \equiv 0 \mod 4$ or $n \equiv r + 2 \equiv 2 \mod 4$.
\end{theorem}
\begin{theorem}\label{froncek1}
	\textnormal{\cite{froncek2011constructing}} Let $n$ be odd, $p \equiv r \equiv 2 \mod 4$, and $G$ be an $r$-regular
	graph with $p$ vertices. Then $G[\bar{K_n}]$ is not distance magic.
\end{theorem}
\begin{theorem}\label{froncek2}
	\textnormal{\cite{froncek2011constructing}} Let $G$ be an arbitrary $r$-regular graph with an odd number
	of vertices and $n$ be an odd positive integer. Then $r$ is even and the graph
	$G[\bar{K_n}]$ is distance magic.
\end{theorem}
\par The following results by Shafiq $et$ $al$. \cite{shafiq2009distance}, discusses the distance magic labeling of disjoint union of $m$ copies of complete multi-partite graphs, $H_{n,p}$, and disjoint union of $m$ copies of product graphs, $C_{p}[\bar{K_{n}}]$.
\begin{theorem}\label{multi}
	\textnormal{\cite{shafiq2009distance}}
	\begin{enumerate}
		\item [\textnormal{(i)}] If $n$ is even or $mnp$ is odd, $m \geq 1; n > 1$ and $p > 1$; then $mH_{n,p}$ has a distance magic labeling.
		\item [\textnormal{(ii)}] If $np$ is odd, $p \equiv 3 \mod 4$, and $m$ is even, then $mH_{n,p}$ does not have a distance magic labeling.
	\end{enumerate}
\end{theorem}
\begin{theorem}\label{multi1}
	\textnormal{\cite{shafiq2009distance}}
	Let $m\geq 1, n>1$ and $p\geq 3.$ $mC_{p}[\bar{K_{n}}]$ has a distance magic labeling if and only if either $n$ is even or $mnp$ is odd or $n$ is odd and $p \equiv 0 \mod 4.$
\end{theorem}
In \cite{shafiq2009distance}, Shafiq $et$ $al$. posted a problem on the graph $mH_{n,p}$.
\begin{problem}
	For the graph $mH_{n,p}$, where $m$ is even, $n$ is odd, $p\equiv 1
	\mod 4,$ and $p > 1$, determine if there is a distance magic labeling.
\end{problem}
Later, Froncek $et$ $al$.\cite{froncek2011constructing} proved the following necessary condition for $mH_{n,p}$.
\begin{theorem}\label{froncek}
	The graph $mH_{n,p}$, where $m$ is even, $n$ is odd, $p\equiv 1
	\mod 4,$ and $p > 1$, is not distance magic.
\end{theorem}

\par
For more details and results, one can refer Arumugam $et$ $al$. \cite{arumugamsurvey2012distance}.
\par From Theorem \ref{oddregular}, one can observe that any odd-regular graph $G$ of order $n$ is not distance magic. But if we label the graph with respect to a different set $S$ of positive integers with $|S|=n$, then $G$ may admit a magic labeling with a magic constant $c'$. See Figure 1.
\begin{figure}[htbp]
	\centering
	\includegraphics[width=40mm]{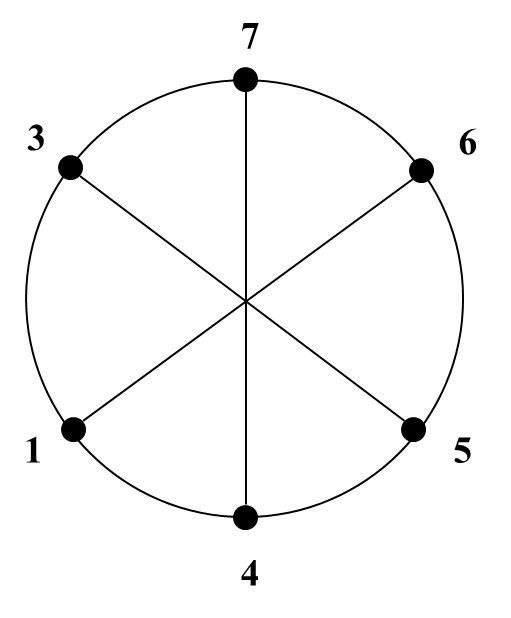}
	\caption{A graph $G$ with $c'=13$ and $S=\{1,3,4,5,6,7\}.$}
	\label{fig1}
\end{figure}
\par Motivated by this fact Godinho $et$ $al.$ \cite{godinho2015s} defined the concept of $S$-magic labeling of a graph.
\begin{definition}
	\textnormal{\cite{godinho2018distance}} Let $G=(V(G),E(G))$ be a graph and let $S$ be a set of positive integers with $|V(G)|=|S|$. Then $G$ is said to be $S$-magic if there exists a bijection
	$f:V(G)\rightarrow S$ satisfying $\sum_{v\in N(u)} f(v) = c$ $($a constant$)$ for every $u\in V(G)$. The constant $c$ is called the $S$-magic constant.
\end{definition}
\begin{definition}
	\textnormal{\cite{godinho2018distance}}
	Let $\alpha(S)$ = max $\{s : s \in S\}$. Let $i(G)$ = min $\alpha(S)$, where the minimum is taken over all sets $S$ for which the graph $G$ admits an $S$-magic labeling. Then $i(G)-|V(G)|$ is called the distance magic index of a graph $G$ and is denoted by $\theta(G)$.
\end{definition}
\par From above definitions, one can observe that a graph $G$ is distance magic if and only if $\theta(G) = 0$ and if $G$ is not $S$-magic for any $S$ with $|V(G)|=|S|$, then $\theta(G)=\infty$.
\par Let $G$ be a graph for which $\theta(G)$ is finite (however so small) and non-zero. Now, a natural question arises that for all such graphs $G$, does there exist an $S$-magic labeling with $\theta(G)=1$?
\par In the following section, we prove some necessary conditions for an $r$-regular $S$-magic graph $G$ to have $\theta(G)=1$. Further, we compute the distance magic index of disjoint union of $m$ copies of $H_{n,p}$ and disjoint union of $m$ copies of $C_{p}[\bar{K_{n}}]$, where $m\geq 1$. Also, for any arbitrary regular graph $G$, we  compute the distance magic index of the graph $G[\bar{K_{n}}]$. In addition to that, we construct twin sets $S$ and $S'$ for the same graph $H_{n,p}$ with $\theta(G)=1$, for which $H_{n,p}$ is both $S$-magic and $S'$-magic with  distinct magic constants. We also discuss the maximum and minimum bounds attained by the magic constant for the graph $H_{n,p}$.
\section{Main results}
If $G$ is a graph with $\theta(G) = 1$, then it is clear that $G$ is $S$-magic for $S = \{1,...,n+1\}\setminus \{a\}$, for at least one $a\in \{1,...,n\}$. We call $a$, the deleted label of $S$.
\par The following results are similar to that of Theorem \ref{oddregular}
and \ref{2mode4}.
\begin{lemma}\label{oddreg2}
	If $G$ is an odd $r$-regular $S$-magic graph with $\theta(G) = 1$, then $a\neq 1$.
\end{lemma}
\begin{proof}
	Assume that $G$ is an $r$-regular graph with $\theta(G) = 1$,  where $r$ is odd. If  $S = \{1,...,n+1\}\setminus \{a\}$ with the $S$-magic constant $c$, then,
	\begin{eqnarray}
		nc&=& r(1+...+n+1)-ra \\
		%nc&=& \frac{r(n+1)(n+2)-2ra}{2} \\
		c&=& \frac{rn+3r}{2}+\frac{r-ra}{n}.
	\end{eqnarray}
	Therefore, if $a = 1$, then $c$ is not an integer, a contradiction.
\end{proof}
\begin{lemma}\label{2mod42}
	If $G$ is an $r$-regular $S$-magic graph with $\theta(G) = 1$ and $r,n \equiv 2 \mod 4$, then $a$ is an even integer, $a\neq 2,n$.
\end{lemma}
\begin{proof}
	Assume that $G$ is an $r$-regular graph with $\theta(G) = 1$ and $r,n \equiv 2 \mod 4$. Let $c$ be the $S$-magic constant of $G$, where $S = \{1,2,...,n+1\}\setminus \{a\}$ and $a$ is an odd integer belonging to $\{1,2,..,n\}$. Let  $r = 4k+2$ and $n = 4k'+2$, with $0<k<k'$. \\
	\textbf{Case 1:} when $a=1$, from eq.(2), we have,
	\begin{eqnarray*}
	c = (2k+1)(4k'+5).
	\end{eqnarray*}
Here $c$ is an odd integer and every vertex is adjacent to odd number of vertices which are labeled with odd integers. Note that, here there are $2k'+1$ such vertices.
Then the graph induced by the vertices having odd label has every vertex of odd degree, a contradiction.\\
	\textbf{Case 2:} When $a= 2q+1$, with $q>0$. Then $rn+3r\equiv 2 \mod 4$ and $r-ra\equiv 0 \mod 4$ and hence $c$ fails to be an integer.\\ %	\begin{eqnarray*}
%		c&=& \frac{rn+3r}{2}+\frac{r-ra}{n}.
%		%c&=& \frac{(4k+2)(4k'+2)+3(4k+2)}{2}+\frac{(4k+2)-(4k+2)(2q+1)}{4k'+2}.\\
%		%c&=& \frac{4(4kk'+5k+2k'+2)+2}{2}-\frac{2(4kq+2q)}{2k'+1}.
%	\end{eqnarray*}
%	But, $rn+3r\equiv 2 \mod 4$ and $r-ra\equiv 0 \mod 4$.
%	%		\begin{eqnarray*}
%	%	%	c&=& \frac{rn+3r}{2}+\frac{r-ra}{n}\\
%	%		c&=& \frac{(4k+2)(4k'+2)+3(4k+2)}{2}+\frac{(4k+2)-(4k+2)(2q+1)}{4k'+2}.\\
%	%		c&=& \frac{4(4kk'+5k+2k'+2)+2}{2}-\frac{2(4kq+2q)}{2k'+1}.
%	%		\end{eqnarray*}
\textbf{Case 3:} When $a = 2$ or $a=n$, $c$ is not an integer and hence the result follows.
\end{proof}
\par The following theorem discusses the distance magic index of the graph, $H_{n,p}, n>1$ and $p>1$. We define the integer-valued function $\alpha$ given by\\
\[ \alpha(j)= \left\{
\begin{array}{ll}
0 &  $ for $ j $  even $ \\
1  &  $ for $ j $ odd, $ \\
\end{array}
\right.\]
and the sets $\Omega_k=\{i: 5\leq i\leq n-1 \textnormal{~and~} i\equiv k\mod 4\}$, where $k\in\{0,1,2,3\}$. Both $\alpha$ and $\Omega_i$'s are used in the next theorem.
\begin{theorem}\label{index1hnp}
	If $G$ is a complete multi-partite graph $H_{n,p}$ with $p$ partitions having $n$ vertices in each partition, then
	\[  \theta(G) = \left\{
	\begin{array}{ll}
	0 &  $ for $n$ even or $n$ and $p$ both odd$ \\
	1 &  $ for $n$ odd and $p$ even.$ \\
	\end{array}
	\right.\]
\end{theorem}
\begin{proof}
	Let $G\cong H_{n,p}$ with $n>1, p>1$. From Theorem \ref{millerHnp}, it is clear that if $n$ is even or when $n$ and $p$ both are odd, then $\theta(G)=0$.
	\par Now, to construct an $(n\times p)$- rectangular matrix $A=(a_{i,j})$ with distinct entries from a set $S$ having column sum $b$~(a constant) is equivalent to find an $S$-magic labeling of $G$ with magic constant $(p-1)b$.
	\par Note that $j^{th}$ column of $A$ can be used to label the vertices of $j^{th}$ partition of $G$ and hence $G$ admits a magic labeling with magic constant $(p-1)b$. In addition, if the entries of $A$ are all distinct and are from $S = \{1,...,np+1\}\setminus\{a\}$, where $a\in\{1,...,np\}$, then $G$ is $S$-magic with $\theta(G)= 1.$
	\par Let $n$ be an odd and $p$ be an even integer.\\\\
	Case 1: If $n = 3$ and $p = 2m, m>0$, then construct $A$ as,
	\begin{center}
		$\begin{pmatrix}
		1  & 2 & 3 & 4 & ... & 2m-3 & 2m-2 & 2m-1 & 2m\\
		3m & 4m & 3m-1 & 4m-1 & ... & 2m+2 & 3m +2 & 2m+1 & 3m+1\\
		6m +1 & 5m & 6m & 5m-1 & ... & 5m+3 & 4m+2 & 5m+2 & 4m+1
		\end{pmatrix}$
	\end{center}
	Note that, the deleted label is $5m+1$ here. One can observe that each column adds up to a constant $9m+2$ and thus, $\theta(H_{3,2m})=1$.\\\\
	Case 2: If $n = 5$ and $p = 2m, m>0$, then construct $A$ as,
	\begin{center}
		\small
		$\begin{pmatrix}
		1  & 2 & 3 & 4 & ... & 2m-3 & 2m-2 & 2m-1 & 2m\\
		3m & 4m & 3m-1 & 4m-1 & ... & 2m+2 & 3m +2 & 2m+1 & 3m+1\\
		6m & 5m & 6m-1 & 5m-1 & ... & 5m+2 & 4m+2 & 5m+1 & 4m+1\\
		7m & 8m & 7m-1 & 8m-1 & ... & 6m+2 & 7m+2 & 6m+1 & 7m+1\\
		9m+2 & 8m+1 & 9m+3 & 8m+2 & ... & 10m & 9m-1 & 10m+1 &  9m
		\end{pmatrix}$
	\end{center}
	%Note that here.
	Here, the deleted label is $9m+1$ and each column adds up to a constant $25m+3$. Therefore, $\theta(H_{5,2m})=1$.\\\\
	Case 3: If $n>5$ is odd and $p = 2m, m>0$, then for each $j\in\{1,...,p\}$, construct $A$ as follows.
	%\[  M_{ij} = \left\{
	%\begin{array}{ll}
	%j &  1\leq j \leq 2m, i=1 \\
	%
	%3m - (\frac{j-1}{2}) & $ for $ j $ odd, $ i=2 \\
	%4m - (\frac{j-2}{2}) & $ for $ j $ even, $ i=2 \\
	%6m - (\frac{j-1}{2}) & $ for $ j $ odd, $ i=3 \\
	%5m - (\frac{j-2}{2}) & $ for $ j $ even, $ i=3 \\
	%7m - (\frac{j-1}{2}) & $ for $ j $ odd, $ i=4 \\
	%8m - (\frac{j-2}{2}) & $ for $ j $ even, $ i=4 \\
	%2mi-2m+(\frac{j+1}{2}) & $ for $ j $ odd, $ i $ odd, $ i\equiv 1 \mod 4, 5\leq i\leq n-1\\
	%2mi-m+(\frac{j}{2}) & $ for $ j $ even, $ i $ odd, $ i\equiv 1 \mod 4, 5\leq i\leq n-1\\
	%2mi-m+(\frac{j+1}{2}) & $ for $ j $ odd, $ i $ even, $ i\equiv 2 \mod 4, 5\leq i\leq n-1\\
	%2mi-2m+(\frac{j}{2}) & $ for $ j $ even, $ i $ even, $ i\equiv 2 \mod 4,5\leq i\leq n-1\\
	%2mi-(\frac{j-1}{2}) & $ for $ j $ odd, $ i $ odd, $ i\equiv 3 \mod 4, 5\leq i\leq n-1\\
	%2mi-m-(\frac{j-2}{2}) & $ for $ j $ even, $ i $ odd, $ i\equiv 3 \mod 4, 5\leq i\leq n-1\\
	%2mi-m-(\frac{j-1}{2}) & $ for $ j $ odd, $ i $ even, $ i\equiv 0 \mod 4, 5\leq i\leq n-1\\
	%2mi-(\frac{j-2}{2}) & $ for $ j $ even, $ i $ even, $ i\equiv 0 \mod 4, 5\leq i\leq n-1\\
	%2mi-m+1+(\frac{j+1}{2}) & $ for $ j $ odd, $ i = n \equiv 1 \mod 4\\
	%2mi-2m+(\frac{j}{2}) & $ for $ j $ even$, i = n \equiv 1 \mod 4\\
	%2mi+1-(\frac{j-1}{2}) & $ for $ j $ odd$, i = n \equiv 3 \mod 4\\
	%2mi-m-(\frac{j-2}{2}) & $ for $ j $ even$, i = n \equiv 3 \mod 4.\\
	%\end{array}
	%\right.\]
	{\small
		\[  a_{i,j} = \left\{
		\begin{array}{ll}
		j & $ for $ i=1 \\
		(2i-1)m-(\frac{j-1}{2})+\alpha(j+1)(m+\frac{1}{2}) & $ for $ i=2,4 \\
		2mi-(\frac{j-1}{2})+\alpha(j+1)(-m+\frac{1}{2}) & $ for $ i=3 \\
		2mi-m+\frac{j}{2}+\alpha(j)(-m+\frac{1}{2}) & $ for $ i\equiv 1 \mod 4, i\in\{5,6,...,n-1\}\\
		2mi-m+\frac{j}{2}+\alpha(j)\frac{1}{2}- \alpha(j+1)m & $ for $ i\equiv 2 \mod 4, i\in\{5,6,...,n-1\}\\
		2mi-(\frac{j-1}{2})+\alpha(j+1)(-m+\frac{1}{2}) & $ for $ i\equiv 3 \mod 4, i\in\{5,6,...,n-1\}\\
		2mi-(\frac{j-1}{2})+\alpha(j)(-m)+\alpha(j+1)\frac{1}{2} & $ for $ i\equiv 0 \mod 4, i\in\{5,6,...,n-1\}\\
		2mi-m+\frac{j}{2}+\alpha(j)(\frac{3}{2})-\alpha(j+1)m & $ for $ i = n \equiv 1 \mod 4\\
		2mi-\frac{j}{2}+\frac{1}{2}-\alpha(j+1)(m+\frac{1}{2}) & $ for $ i = n \equiv 3 \mod 4\\
		\end{array}
		\right.\]
	}
	\noindent Therefore, $m(2n-1)+1$ is the deleted label in this case.\\\\
	Subcase 1: If $n\equiv 1\mod 4$, then $n-5\equiv 0\mod 4$. Let $n= 4q+5$, where $q\geq 1$. \\ Now for any fixed odd $j$, the $j^{th}$ column sum in $A$ is,
	\begin{eqnarray*}
		\sum_{i=1}^{n}a_{i,j} = \sum_{i=1}^{4}a_{i,j}+\sum_{i\in\Omega_1}a_{i,j}+\sum_{i\in\Omega_2}a_{i,j}+\sum_{i\in\Omega_3}a_{i,j}+\sum_{i\in\Omega_0} a_{i,j}+ \biggl(2mn-m+\frac{j}{2}+\frac{3}{2}\biggr)
	\end{eqnarray*}
	\noindent =$j+3m-(\frac{j-1}{2})+6m-(\frac{j-1}{2})+7m-(\frac{j-1}{2})+ \sum_{k=1}^q \biggl(2m(4k+1)-2m+\frac{j+1}{2}\biggr)+\\\sum_{k=1}^q \biggl(2m(4k+2)-m+\frac{j+1}{2}\biggr)+ \sum_{k=1}^q \biggl(2m(4k+3)-(\frac{j-1}{2})\biggr)+\\ \sum_{k=1}^q \biggl(2m(4k+4)-m-(\frac{j-1}{2})\biggr)+2mn-m+1+\frac{j+1}{2}\\
	= 15m+32mq+16mq^2+2mn+2q+3
	= \frac{n^2p+n+1}{2}$.\\\\
	Similarly, for any fixed even $j$, the $j^{th}$ column sum in $A$ is,
	\begin{eqnarray*}
		\sum_{i=1}^{n}a_{i,j} = \sum_{i=1}^{4}a_{i,j}+\sum_{i\in\Omega_1}a_{i,j}+\sum_{i\in\Omega_2}a_{i,j}+\sum_{i\in\Omega_3}a_{i,j}+\sum_{i\in\Omega_0}a_{i,j}+ \biggl(2mn-2m+\frac{j}{2}\biggr)
	\end{eqnarray*}
	$=j+4m-(\frac{j-2}{2})+5m-(\frac{j-2}{2})+8m-(\frac{j-2}{2})+ \sum_{k=1}^q \biggl(2m(4k+1)-m+\frac{j}{2}\biggr)+\\\sum_{k=1}^q \biggl(2m(4k+2)-2m+\frac{j}{2}\biggr)+\sum_{k=1}^q\biggl(2m(4k+3)-m-(\frac{j-2}{2})\biggr)+\\ \sum_{k=1}^q\biggl(2m(4k+4)-(\frac{j-2}{2})\biggr)+2mn-2m+\frac{j}{2}\\ = 15m+32mq+16mq^2+2mn+2q+3 = \frac{n^2p+n+1}{2}$.\\\\
	Subcase 2: if $n \equiv 3 \mod 4$, then $n-5 \equiv 2 \mod 4$. Let $n = 4q+3 $ where $q\geq0$. \\ Now, for any fixed odd $j$, the $j^{th}$ column sum in $A$ is,
	\begin{eqnarray*}
		\sum_{i=1}^{n}a_{i,j} = \sum_{i=1}^{4}a_{i,j}+\sum_{i\in\Omega_1}a_{i,j}+\sum_{i\in\Omega_2}a_{i,j}+\sum_{i\in\Omega_3}a_{i,j}+\sum_{i\in\Omega_0}a_{i,j}+\biggl(2mn-\frac{j}{2}+\frac{3}{2}\biggr)
	\end{eqnarray*}
	$=j+3m-(\frac{j-1}{2})+6m-(\frac{j-1}{2})+7m-(\frac{j-1}{2})+ \sum_{k=1}^{q+1} \biggl(2m(4k+1)-2m+\frac{j+1}{2}\biggr)+\\ \sum_{k=1}^{q+1} \biggl(2m(4k+2)-m+\frac{j+1}{2}\biggr)+\sum_{k=1}^q \biggl(2m(4k+3)-(\frac{j-1}{2})\biggr)+\\ \sum_{k=1}^q \biggl(2m(4k+4)-m-(\frac{j-1}{2})\biggr)+2mn+1-(\frac{j-1}{2}) \\ = 35m+48mq+16mq^2+2mn+2q+4 = \frac{n^2p+n+1}{2}$.\\\\
	Similarly, for any fixed even $j$, the $j^{th}$ column sum in $A$ is,
	\begin{eqnarray*}
		\sum_{i=1}^{n}a_{i,j} = \sum_{i=1}^{4}a_{i,j}+\sum_{i\in\Omega_1}a_{i,j}+\sum_{i\in\Omega_2}a_{i,j}+\sum_{i\in\Omega_3}a_{i,j}+\sum_{i\in\Omega_0}+\biggl(2mn-m-\frac{j}{2}+1\biggr)
	\end{eqnarray*}
	$=j+4m-(\frac{j-2}{2})+5m-(\frac{j-2}{2})+8m-(\frac{j-2}{2})+ \sum_{k=1}^{q+1} \biggl(2m(4k+1)-m+\frac{j}{2}\biggr)+\\ \sum_{k=1}^{q+1} \biggl(2m(4k+2)-2m+\frac{j}{2}\biggr)+ \sum_{k=1}^q \biggl(2m(4k+3)-m-(\frac{j-2}{2})\biggr)+\\ \sum_{k=1}^q \biggl(2m(4k+4)-(\frac{j-2}{2})\biggr)+2mn-m-(\frac{j-2}{2})\\ = 35m+48mq+16mq^2+2mn+2q+4 = \frac{n^2p+n+1}{2}$.\\\\
	\par Since the sum of the entries in each column of $A$ is $\frac{n^2p+n+1}{2}$ for odd $n>5$, $H_{n,2m}$ is $S$-magic with magic constant $\frac{n^2p+n+1}{2}(p-1)$ and $\theta(H_{n,2m})=1$.
\end{proof}
\vskip 0.5 cm
\begin{theorem}
	If $G\cong H_{n,p}$ is an $S$-magic graph with $\theta(G) = 1$ and $S$-magic constant $\frac{n^2p+n+1}{2}(p-1)$, then there exists a set $S'$ such that $G$ is an $S'$-magic graph with $\theta(G) = 1$ and $S'$-magic constant $\frac{n^2p+3n-1}{2}(p-1).$
\end{theorem}
\begin{proof}
	For every $S$-magic graph $G\cong H_{n,p}$ with $\theta(G) = 1$, one can obtain the corresponding rectangular matrix $A=(a_{i,j})$ associated with $G$ by Theorem \ref{index1hnp}.
	\par Define a new $(n\times p)$- rectangular matrix $A'=(a'_{i,j})$ with entries,
	\begin{eqnarray}\label{eq1}
	a'_{i,j} = (np+2) - a_{i,j} \textnormal{~for all}~i~ \textnormal{and}~j.
	\end{eqnarray}
	By Theorem \ref{index1hnp}, it is clear that the entries in $A$ belong to the set $\{1,...,np+1\}\setminus\{np+1-\frac{p}{2}\}$, which sum up to $\frac{n^2p^2+p(n+1)}{2}$ and is divisible by $p$. Hence the magic constant is $\frac{n^2p+n+1}{2}(p-1)$.\par
	Now using (\ref{eq1}), define the new set $S'= S\cup\{np+1-\frac{p}{2}\}\setminus \{\frac{p}{2}+1\}$ and the sum of all the entries in $A' = np(np+2)-(\frac{n^2p^2+p(n+1)}{2}) = \frac{n^2p^2+3np-p}{2},$ which is divisible by $p$. Therefore, we obtain the magic constant as $\frac{n^2p+3n-1}{2}(p-1)$.
\end{proof}
\vskip 0.5 cm
The rectangular matrices $A$ and $A'$ associated with $H_{5,6}$ are given below,\\\\
\begin{minipage}{7 cm}
	\begin{center}
		$A$ = $\begin{pmatrix}
		1 & 2  & 3 & 4  & 5 & 6 \\
		9 & 12 & 8 & 11 & 7 & 10 \\
		18 & 15 & 17 & 14 & 16 & 13	\\
		21 & 24 & 20 & 23 & 19 & 22	\\
		29 & 25 & 30 & 26 & 31 & 27 \\			
		\end{pmatrix}$
	\end{center}
\end{minipage}
\begin{minipage}{7 cm}
	\begin{center}
		$A'$ = $\begin{pmatrix}
		31 & 30  & 29 & 28  & 27 & 26 \\
		23 & 20 & 24 & 21 & 25 & 22 \\
		14 & 17 & 15 & 18 & 16 & 19	\\
		11 & 8 & 12 & 9 & 13 & 10	\\
		3 & 7 & 2 & 6 & 1 & 5 \\			
		\end{pmatrix}$
	\end{center}
\end{minipage}
\vskip 0.5 cm
\par
Here, the sum of the entries in each column of $A$ and $A'$ are 78 and 82 respectively. Then, $H_{5,6}$ is $S$-magic with magic constant 390 and $S'$-magic with magic constant 410.
\begin{figure}[htbp]
	\centering
	\includegraphics[width=140mm]{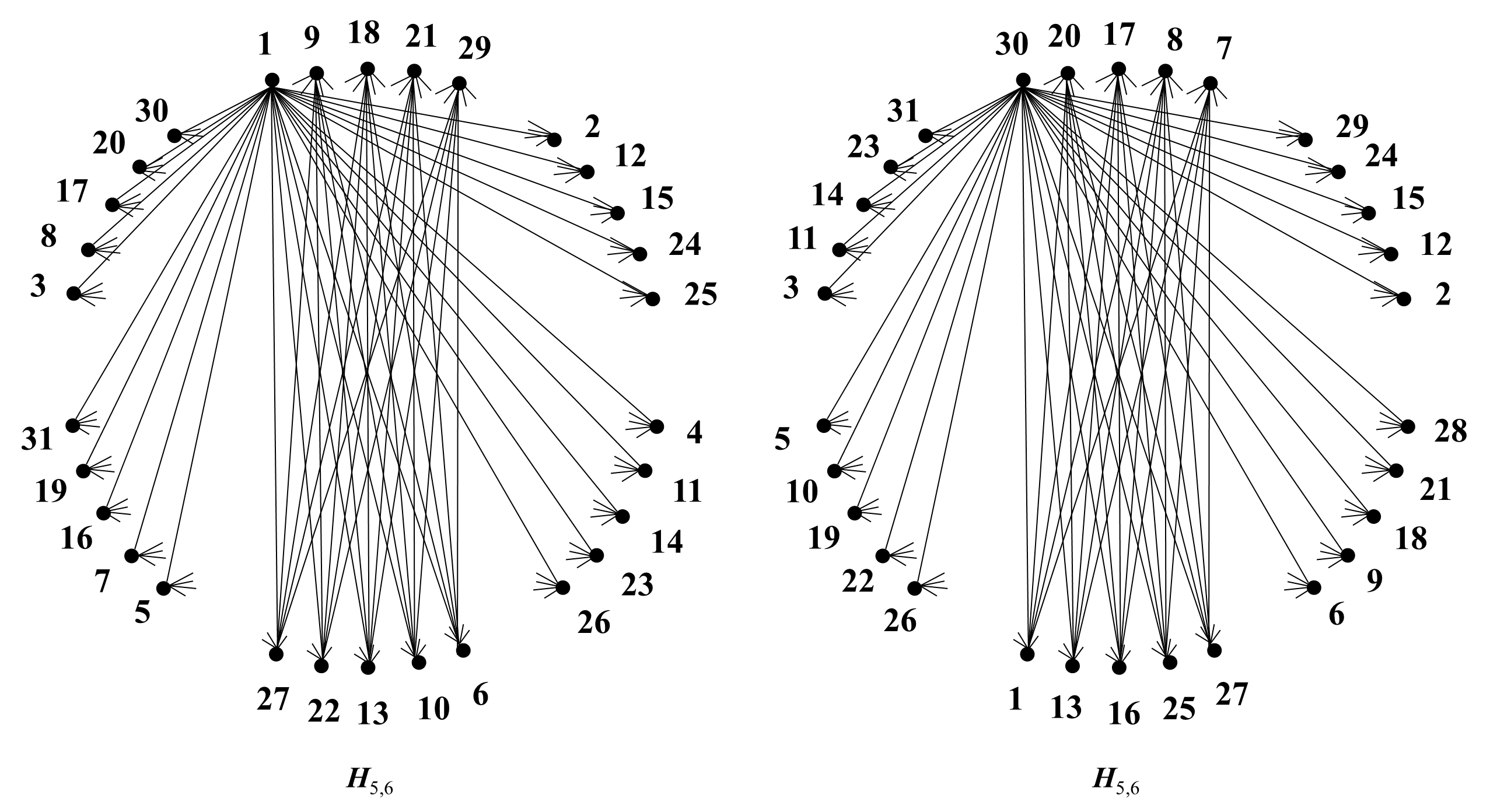}
	\caption{$H_{5,6}$ with $S$-magic constant $390$ and $S'$ magic constant 410.}
	\label{fig3}
\end{figure}
%\begin{figure}[t]
%	\centering
%	\includegraphics[width=80mm]{H562}
%	\caption{$H_{5,6}$ with $S'$-magic constant $410$.}
%	\label{fig4}
%\end{figure}
\newline
\noindent Now the following result is immediate.
\begin{lemma}\label{2magicconst}
	If $G$ is an $r$-regular graph with $\theta(G)=1$ and with $S$-magic constant $c$, then
	\begin{equation*}
	\frac{nr+r}{2}+\frac{r}{n} \leq c \leq \frac{nr+3r}{2}.
	\end{equation*}
	\begin{proof}
		The proof is obtained from Lemma \ref{oddreg2} by substituting $a = 1$ and $a = n$ for $c$.
	\end{proof}
\end{lemma}
\begin{observation}\label{hnp}
	If $G\cong H_{n,p}$ is a graph with $\theta(G)=1$ and $S$-magic constant $c$, then
	\begin{eqnarray*}
		\frac{n^2p+n+1}{2}(p-1)\leq c \leq \frac{n^2p+3n-1}{2}(p-1)
	\end{eqnarray*}	
\end{observation}
\qed
\par The lower and upper bounds in Observation \ref{hnp} are tight when one compares with Lemma \ref{2magicconst}. It is noticed that if $S=\{1,...,np+1\}\setminus \{a\}$, which confirms that $H_{n,p}$ is $S$-magic, then the sum of all the entries in $S$ is divisible by $p$. Therefore, the highest $a$ that can be removed to get a multiple of $p$ is $np+1-\frac{p}{2}$ and the lowest $a$ that can be removed to get a multiple of $p$ is $\frac{p}{2}+1.$ Hence the result follows.
\begin{lemma}\label{rectangleconstruct}
	Let $B$ be an $(n\times p)$-rectangular matrix with distinct entries from the set $\{1,2,..,np+1\}\setminus \{a\},$ where $a\in\{1,2,...,np\}$ having column sum $s$. If there exists an integer $m\geq 1$, $m|p$, then there exixts $m$, $(n\times t)$-rectangular matrices, $B_{m}, (1\leq m \leq t)$, having column sum $s$.
\end{lemma}
\begin{proof}
	Consider the $(n\times mt)$-rectangular matrix $B$ with distinct entries from the set $\{1,2,..,np+1\}\setminus \{a\},$ where $a\in\{1,2,...,np\}$ and having column sum $s$.
	\par Construct an $(n\times t)$-rectangular matrix, $B_{1}$ by choosing any $t$ distinct columns of $B$ and update the $B$ matrix by replacing all the entries in the newly chosen $t$ columns with $0's$. Now the updated $B$ matrix will have exactly $(m-1)t$ nonzero columns and $t$ columns having all zero entries.
	\par Now, repeat the process to obtain the next matrix $B_{2}$ by choosing any $t$ non-zero columns from the remaining $(m-1)t$ columns and update the $B$ matrix in the same manner as in first step. Now repeatedly apply the above technique to obtain the remaining $m-2$ matrices, $B_{i},(3\leq i \leq m)$, until the matrix $B$ becomes an zero matrix.
\end{proof}
\par From Theorem \ref{multi}, it is observed that in both the cases when $n$ is odd, $p$ is even and when $np$ is odd, $p\equiv 3 \mod 4$ and $m$ is even, $\theta(mH_{n,p})\neq0$. The following theorem computes the distance magic index of $mH_{n,p}$ for above cases.
\begin{theorem}\label{multi6}
	If $n>1,p>1,m\geq 1$, then\\
	\[  \theta(mH_{n,p}) = \left\{
	\begin{array}{ll}
	0 &  $ for $n$ even or $mnp$ is odd,$\\
	1 &  $ otherwise.$
	\end{array}
	\right.\]
\end{theorem}
\begin{proof}
	Using Theorem \ref{multi}, it is clear that $\theta(mH_{n,p})=0$, when either $n$ is even or $mnp$ is odd and $\theta(mH_{n,p})\neq 0$, when either $np$ is odd, $p \equiv 3 \mod 4$, and $m$ is even. On the other hand, by Theorem \ref{froncek}, one can conclude that $\theta(mH_{n,p})\neq 0$, when $m$ is even, $n$ is odd, $p\equiv 1
	\mod 4,$ and $p > 1$.\\
	For all the remaining cases, use Theorem \ref{index1hnp} to construct the rectangular matrix $A$ associated with the graph $H_{n,mp}$.
	Now using Lemma \ref{rectangleconstruct}, construct the $(n\times p)$-matrices $B_{k},$ for $k\in\{1,...,m\}$
%	\par For $k\in\{1,...,m\}$, define the $(n\times p)$-matrices $B_{k}=(b_{i,j})$ as  $b_{i,j} = a_{i,s}$ for all  $i\in\{ 1,...,n\}$ and for all $s\in \{(k-1)p+1,..., kp\}$.
	Here, each $B_{k}$ forms the matrix associated with the $k^{th}$ copy of $H_{n,p}$ and hence we obtain an $S$-magic labeling of $mH_{n,p}$ with $c = \frac{n^2mp+n+1}{2}(p-1)$. Therefore, $\theta(mH_{n,p})=1.$
\end{proof}
Theorem \ref{multi1} confirms that if $n$ is even or $mnp$ is odd or $n$ is odd and $p \equiv 0 \mod 4$, then $\theta(mC_{p}[\bar{K_{n}}])=0.$
Now the remaining cases are given below. \\
\textbf{Case 1:} $n$ is odd, $m$ is even, $p\equiv 2 \mod 4.$\\
\textbf{Case 2:} $n$ is odd, $m$ is odd, $p\equiv 2 \mod 4.$\\
\textbf{Case 3:} $n$ is odd, $m$ is even, $p$ is odd.
\par The following theorem determines the distance magic index of the graph $mC_{p}[\bar{K_{n}}]$ for all the above mentioned three cases.
\begin{theorem}
	Let $m\geq 1, n>1$ and $p\geq3$, then\\
	\[  \theta(mC_{p}[\bar{K_{n}}]) = \left\{
	\begin{array}{ll}
	0 &  $ if $n$ is~even or $mnp$ is odd or $n$ is odd, $p\equiv 0 \textnormal{~mod~} 4,\\
	1 &  $ otherwise.$
	\end{array}
	\right.\]
\end{theorem}
\begin{proof}
	Let $G\cong mC_{p}[\bar{K_{n}}]$. From Theorem \ref{multi1}, it is clear that $\theta(G)=0$, when $n$ even or $mnp$ is odd or $n$ is odd and $p\equiv 0 \mod 4.$\par
	Now, for all the remaining cases, using Theorem \ref{index1hnp} construct the matrix $A$ associated with the graph $H_{n,mp}$ and use $A$ in Lemma \ref{rectangleconstruct} to construct the $m$ rectangular matrices associated with $m$ copies of graph $C_{p}[\bar{K_{n}}]$. Hence, we obtain a $S$-magic labeling of $G$ with $c = n^2mp+n+1$ and hence $\theta(G)=1.$
\end{proof}
Let $G$ be an $r$-regular graph on $p$ vertices. From Theorem \ref{froncek3}, for the graph $G[\bar{K_{n}}]$, if $n$ is odd, $r$ is even and $p$ is even except when $p\equiv r \equiv 2\mod 4$, then $\theta(G[\bar{K_{n}}] = 0.$
%From Theorem \ref{millerregularcomp} it is clear that the graph $G[\bar{K_{n}}]$ is distance magic if $n$ is even. Also from Theorem \ref{froncek2} we have that if $r$ is even, $p$ is odd and $n$ is odd, $\theta(G[\bar{K_{n}}])=0$ and from Theorem \ref{froncek1}, if $r\equiv p\equiv 2 \mod 4$ and $n$ is odd, then $\theta(G[\bar{K_{n}}])\neq0$.
The following theorem computes the distance magic index of the graph $G[\bar{K_{n}}]$.
\begin{theorem}
	Let $G$ be an $r$-regular graph on $p$ vertices. Then,
		\[ \theta(G[\bar{K_{n}}]) = \left\{
	\begin{array}{ll}
	0 &  $ if $n$ is even$~or~$n,p$~are~odd$, $r$ is even$,\\
	1 &  $ if $n,r$ are odd~or~$n$~is~odd, $r~\equiv p\equiv2~mod~4\\
	0 &  $ otherwise$.
	\end{array}
	\right.\]
\end{theorem}
\begin{proof}
	Let $G$ be a graph on $p$ vertices $v_{1},...,v_{p}$ and let $V_{i} = \{v_{i}^{1},...,v_{i}^n\}$ be set the vertices of $G[\bar{K_{n}}]$ that replace the vertex $v_{i}$ of $G$ for all $i = 1,...,p$. Note that here $V(G[\bar{K_{n}}])=\bigcup_{i=1}^{p}V_{i}.$  \par When $n$ is even, by Theorem \ref{millerregularcomp}, $\theta(G[\bar{K_{n}}])=0$ and when $n$ is odd, $p$ is odd and $r$ is even, by Theorem \ref{froncek2}, $\theta(G[\bar{K_{n}}])=0$. Further, when $n$ is odd and $p\equiv r \equiv 2\mod 4$, then by Theorem \ref{froncek1}, $\theta(G[\bar{K_{n}}])\neq0$. Also when $n$ is odd and $r$ is odd, by Theorem \ref{oddregular}, $\theta(G[\bar{K_{n}}])\neq0$. Further for all the other cases $\theta(G[\bar{K_{n}}]=0$ by Theorem \ref{froncek3}. Now for both the cases when  $\theta(G[\bar{K_{n}}])\neq0$, use Theorem \ref{index1hnp}, to construct the rectangular matrix $A$ associated with the graph $H_{n,p}$ and use the $i^{th}$ column of $A$ to label the set of vertices, $V_{i}$, for all $i=1,2,..,p$. Hence, we obtain a S-magic labeling of $G[\bar{K_{n}}]$, with  $c = r\bigl(\frac{n^2p+n+1}{2}\bigr)$. Therefore we obtain that $\theta(C_{p}[\bar{K_{n}}])=1.$
\end{proof}

%----------------------------------------------
\section{Conclusion}
In this paper, the distance magic index of disjoint union of $m$ copies of $H_{n,p}$ and disjoint union of $m$ copies of $C_{p}[\bar{K_{n}}]$ are computed and few necessary conditions are derived for a regular graph $G$ for which $\theta(G)$ is  $1$. The paper establishes a technique to construct a new set of labels from an existing one in such a way that both magic constants are distinct. Further, the lower and upper bounds of magic constant of a regular graph $G$ with $\theta(G)=1$, are also determined.
\bibliographystyle{elsarticle-num.bst}
\bibliography{bibiliography_list}
\end{document}